\documentclass[12pt]{amsart}
\usepackage[margin=1in]{geometry}
\usepackage{amsmath,amsfonts,amsthm,amssymb,bbm}
\usepackage{graphicx,color,dsfont}
\usepackage{enumitem}
\usepackage{fourier}

\begin{document}

\newtheorem{theorem}{Theorem}
\newtheorem{lemma}{Lemma}
\newtheorem{proposition}{Proposition}
\newtheorem{rmk}{Remark}
\newtheorem{example}{Example}
\newtheorem{exercise}{Exercise}
\newtheorem{definition}{Definition}
\newtheorem{corollary}{Corollary}
\newtheorem{notation}{Notation}
\newtheorem{claim}{Claim}

\newtheorem{dif}{Definition}

 \newtheorem{thm}{Theorem}[section]
 \newtheorem{cor}[thm]{Corollary}
 \newtheorem{lem}[thm]{Lemma}
 \newtheorem{prop}[thm]{Proposition}
 \theoremstyle{definition}
 \newtheorem{defn}[thm]{Definition}
 \theoremstyle{remark}
 \newtheorem{rem}[thm]{Remark}
 \newtheorem*{ex}{Example}
 \numberwithin{equation}{section}

\newcommand{\vertiii}[1]{{\left\vert\kern-0.25ex\left\vert\kern-0.25ex\left\vert #1
    \right\vert\kern-0.25ex\right\vert\kern-0.25ex\right\vert}}

\newcommand{\R}{{\mathbb R}}
\newcommand{\C}{{\mathbb C}}
\newcommand{\U}{{\mathcal U}}
\newcommand{\norm}[1]{\left\|#1\right\|}
\renewcommand{\(}{\left(}
\renewcommand{\)}{\right)}
\renewcommand{\[}{\left[}
\renewcommand{\]}{\right]}
\newcommand{\f}[2]{\frac{#1}{#2}}
\newcommand{\im}{i}
\newcommand{\cl}{{\mathcal L}}
\newcommand{\ck}{{\mathcal K}}

\newcommand{\al}{\alpha}
\newcommand{\vro}{\varrho}
\newcommand{\be}{\beta}
\newcommand{\wh}[1]{\widehat{#1}}
\newcommand{\ga}{\gamma}
\newcommand{\Ga}{\Gamma}
\newcommand{\de}{\delta}
\newcommand{\ben}{\beta_n}
\newcommand{\De}{\Delta}
\newcommand{\ve}{\varepsilon}
\newcommand{\ze}{\zeta}
\newcommand{\Th}{\Theta}
\newcommand{\ka}{\kappa}
\newcommand{\la}{\lambda}
\newcommand{\laj}{\lambda_j}
\newcommand{\lak}{\lambda_k}
\newcommand{\La}{\Lambda}
\newcommand{\si}{\sigma}
\newcommand{\Si}{\Sigma}
\newcommand{\vp}{\varphi}
\newcommand{\om}{\omega}
\newcommand{\Om}{\Omega}
\newcommand{\ra}{\rightarrow}

\newcommand{\ro}{{\mathbf R}}
\newcommand{\rn}{{\mathbf R}^n}
\newcommand{\rd}{{\mathbf R}^d}
\newcommand{\rmm}{{\mathbf R}^m}
\newcommand{\rone}{\mathbb R}
\newcommand{\rtwo}{\mathbf R^2}
\newcommand{\rthree}{\mathbf R^3}
\newcommand{\rfour}{\mathbf R^4}
\newcommand{\ronen}{{\mathbf R}^{n+1}}
\newcommand{\ku}{\mathbf u}
\newcommand{\kw}{\mathbf w}
\newcommand{\kf}{\mathbf f}
\newcommand{\kz}{\mathbf z}

\newcommand{\N}{\mathbf N}

\newcommand{\tn}{\mathbf T^n}
\newcommand{\tone}{\mathbf T^1}
\newcommand{\ttwo}{\mathbf T^2}
\newcommand{\tthree}{\mathbf T^3}
\newcommand{\tfour}{\mathbf T^4}

\newcommand{\zn}{\mathbf Z^n}
\newcommand{\zp}{\mathbf Z^+}
\newcommand{\zone}{\mathbf Z^1}
\newcommand{\zz}{\mathbf Z}
\newcommand{\ztwo}{\mathbf Z^2}
\newcommand{\zthree}{\mathbf Z^3}
\newcommand{\zfour}{\mathbf Z^4}

\newcommand{\hn}{\mathbf H^n}
\newcommand{\hone}{\mathbf H^1}
\newcommand{\htwo}{\mathbf H^2}
\newcommand{\hthree}{\mathbf H^3}
\newcommand{\hfour}{\mathbf H^4}

\newcommand{\cone}{\mathbf C^1}
\newcommand{\ctwo}{\mathbf C^2}
\newcommand{\cthree}{\mathbf C^3}
\newcommand{\cfour}{\mathbf C^4}
\newcommand{\dpr}[2]{\langle #1,#2 \rangle}

\newcommand{\sn}{\mathbf S^{n-1}}
\newcommand{\sone}{\mathbf S^1}
\newcommand{\stwo}{\mathbf S^2}
\newcommand{\sthree}{\mathbf S^3}
\newcommand{\sfour}{\mathbf S^4}

\newcommand{\lp}{L^{p}}
\newcommand{\lppr}{L^{p'}}
\newcommand{\lqq}{L^{q}}
\newcommand{\lr}{L^{r}}
\newcommand{\echi}{(1-\chi(x/M))}
\newcommand{\chip}{\chi'(x/M)}

\newcommand{\wlp}{L^{p,\infty}}
\newcommand{\wlq}{L^{q,\infty}}
\newcommand{\wlr}{L^{r,\infty}}
\newcommand{\wlo}{L^{1,\infty}}

\newcommand{\lprn}{L^{p}(\rn)}
\newcommand{\lptn}{L^{p}(\tn)}
\newcommand{\lpzn}{L^{p}(\zn)}
\newcommand{\lpcn}{L^{p}(\cn)}
\newcommand{\lphn}{L^{p}(\cn)}

\newcommand{\lprone}{L^{p}(\rone)}
\newcommand{\lptone}{L^{p}(\tone)}
\newcommand{\lpzone}{L^{p}(\zone)}
\newcommand{\lpcone}{L^{p}(\cone)}
\newcommand{\lphone}{L^{p}(\hone)}

\newcommand{\lqrn}{L^{q}(\rn)}
\newcommand{\lqtn}{L^{q}(\tn)}
\newcommand{\lqzn}{L^{q}(\zn)}
\newcommand{\lqcn}{L^{q}(\cn)}
\newcommand{\lqhn}{L^{q}(\hn)}

\newcommand{\lo}{L^{1}}
\newcommand{\lt}{L^{2}}
\newcommand{\li}{L^{\infty}}
\newcommand{\beqn}{\begin{eqnarray*}}
\newcommand{\eeqn}{\end{eqnarray*}}
\newcommand{\pplus}{P_{Ker[\cl_+]^\perp}}

\newcommand{\co}{C^{1}}
\newcommand{\ci}{C^{\infty}}
\newcommand{\coi}{C_0^{\infty}}

\newcommand{\ca}{\mathcal A}
\newcommand{\cs}{\mathcal S}
\newcommand{\cm}{\mathcal M}
\newcommand{\cf}{\mathcal F}
\newcommand{\cb}{\mathcal B}
\newcommand{\ce}{\mathcal E}
\newcommand{\cd}{\mathcal D}
\newcommand{\cn}{\mathcal N}
\newcommand{\cz}{\mathcal Z}
\newcommand{\crr}{\mathbf R}
\newcommand{\cc}{\mathcal C}
\newcommand{\ch}{\mathcal H}
\newcommand{\cq}{\mathcal Q}
\newcommand{\cp}{\mathcal P}
\newcommand{\cx}{\mathcal X}
\newcommand{\eps}{\epsilon}

\newcommand{\pv}{\textup{p.v.}\,}
\newcommand{\loc}{\textup{loc}}
\newcommand{\intl}{\int\limits}
\newcommand{\iintl}{\iint\limits}
\newcommand{\dint}{\displaystyle\int}
\newcommand{\diint}{\displaystyle\iint}
\newcommand{\dintl}{\displaystyle\intl}
\newcommand{\diintl}{\displaystyle\iintl}
\newcommand{\liml}{\lim\limits}
\newcommand{\suml}{\sum\limits}
\newcommand{\ltwo}{L^{2}}
\newcommand{\supl}{\sup\limits}
\newcommand{\df}{\displaystyle\frac}
\newcommand{\p}{\partial}
\newcommand{\Ar}{\textup{Arg}}
\newcommand{\abssigk}{\widehat{|\si_k|}}
\newcommand{\ed}{(1-\p_x^2)^{-1}}
\newcommand{\tT}{\tilde{T}}
\newcommand{\tV}{\tilde{V}}
\newcommand{\wt}{\widetilde}
\newcommand{\Qvi}{Q_{\nu,i}}
\newcommand{\sjv}{a_{j,\nu}}
\newcommand{\sj}{a_j}
\newcommand{\pvs}{P_\nu^s}
\newcommand{\pva}{P_1^s}
\newcommand{\cjk}{c_{j,k}^{m,s}}
\newcommand{\Bjsnu}{B_{j-s,\nu}}
\newcommand{\Bjs}{B_{j-s}}
\newcommand{\Ly}{\cl_+i^y}
\newcommand{\dd}[1]{\f{\partial}{\partial #1}}
\newcommand{\czz}{Calder\'on-Zygmund}
\newcommand{\chh}{\mathcal H}

\newcommand{\lbl}{\label}
\newcommand{\beq}{\begin{equation}}
\newcommand{\eeq}{\end{equation}}
\newcommand{\beqna}{\begin{eqnarray*}}
\newcommand{\eeqna}{\end{eqnarray*}}
\newcommand{\bp}{\begin{proof}}
\newcommand{\ep}{\end{proof}}
\newcommand{\bprop}{\begin{proposition}}
\newcommand{\eprop}{\end{proposition}}
\newcommand{\bt}{\begin{theorem}}
\newcommand{\et}{\end{theorem}}
\newcommand{\bex}{\begin{Example}}
\newcommand{\eex}{\end{Example}}
\newcommand{\bc}{\begin{corollary}}
\newcommand{\ec}{\end{corollary}}
\newcommand{\bcl}{\begin{claim}}
\newcommand{\ecl}{\end{claim}}
\newcommand{\bl}{\begin{lemma}}
\newcommand{\el}{\end{lemma}}
\newcommand{\dea}{(-\De)^\be}
\newcommand{\naa}{|\nabla|^\be}
\newcommand{\cj}{{\mathcal J}}
\newcommand{\ubb}{{\mathbf u}}

\title[On the stability of the compacton waves for the degenerate KdV  
	 and NLS]{On the stability of the compacton waves for the degenerate KdV  
	and NLS models}

\author{Sevdzhan Hakkaev}
\address{
	Faculty of Arts and Science, Istanbul Aydin University, Istanbul, Turkey; 
	Faculty of Mathematics and Informatics, Shumen University, 9712
	Shumen, Bulgaria}\email{s.hakkaev@shu.bg}
\author[Abba Ramadan]{\sc Abba Ramadan}
\address{ Department of Mathematics,
University of Kansas,
1460 Jayhawk Boulevard,  Lawrence KS 66045--7523, USA}
\email{aramadan@ku.edu}
\author[Atanas G. Stefanov]{\sc Atanas G. Stefanov}
\address{ Department of Mathematics,
	University of Alabama - Birmingham, 
	University Hall, Room 4005, 
	1402 10th Avenue South
	Birmingham AL 35294-1241
	 }
\email{stefanov@uab.edu}



\date{\today}
 
\begin{abstract}
In this paper,  we consider the degenerate semi-linear Schr\"odinger and Korteweg-deVries equations in one spatial dimension. We construct special solutions of the two models, namely  standing wave solutions of NLS and traveling waves, which turn out to have compact support, compactons. We show that the compactons are unique bell-shaped solutions of the corresponding PDE's and for appropriate variational problems as well. 
 We provide a complete spectral characterization of such waves, for all values of $p$. Namely, we show that all waves are spectrally stable for $2<p\leq 8$,  while a single mode 
instability occurs for $p>8$.  This extends previous work of Germain, Harrop-Griffits and Marzuola, \cite{GM1}, who have previously established orbital stability for some specific waves, in the range  $p<8$. 
 
\end{abstract}

\thanks{Ramadan is partially supported by a  graduate research assistantship under NSF-DMS \# 1614734. Stefanov is partially supported by NSF-DMS \# 1908626.} 

\maketitle

\section{Introduction}
In this paper, we shall be interested in some aspects of the dynamics of dispersive equations, driven by  degenerate dispersion. In order to fix ideas, we settle on the one dimensional case on the line, and recall the standard dispersive models. More specifically, for $p>2$, consider the (generalized) Korteweg de Vries equation 
\begin{equation}
\label{01} 
u_t+u_{xxx}+\p_x(|u|^{p-2} u)=0, \ \ u:\rone\times \rone\to \rone
\end{equation}
and the non-linear Schr\"odinger equation 
\begin{equation}
\label{02} 
i u_t+u_{xx}+ |u|^{p-2} u=0, \ \ u:\rone\times \rone\to {\mathbb C}.
\end{equation}
 The Cauchy problem for these models is well-understood, even in low regularity setting. In particular, and for classical solutions, the solutions conserve   energy 
$$
E[u(t)]=\f{1}{2} \int_\rone |u'|^2 dx - \f{1}{p}\int_\rone    |u|^p dx=E[u_0]
$$
and mass 
$$
M[u(t)]=\int_\rone  |u|^2 dx=M[u_0].
$$
The solitary waves of these models, namely the standing waves $e^{i \om t} \phi$ of 
 \eqref{02} and the traveling waves $\phi(x-\om t)$ of \eqref{01}, have received ample attention in the literature. It turns out that they are unique (i.e. for each $\om>0, p>2$, there is an unique function $\phi$ with this property). Their stability is also a classical fact by now (see for example \cite{CL}, but also \cite{CT3, CT4, JC})  - namely these waves are spectrally stable for all $\om>0$, when $2<p<6$,  while they become spectrally unstable for $p>6$, again for all values of $\om>0$. 
 
In this paper, we  investigate   the degenerate KdV and NLS. 
More specifically, for $p>2$, the degenerate KdV is given by 
  \begin{equation}
  \label{10} 
  u_t+ \p_x(u\p_x(u\p_x u)+|u|^{p-2} u)=0, \ \ u:\rone\times \rone\to \rone.
  \end{equation}
The degenerate KdV model was introduced in \cite{Ros, RosHy}, with the main interest in the compacton traveling waves. A  more general Hamiltonian version of this problem, was proposed in \cite{CSS, ZC} and subsequently in  \cite{MCCS}.  

The degenerate NLS takes the form 
  \begin{equation}
  \label{20} 
  i u_t+ \bar{u} \p_x(u \p_x u)+|u|^{p-2} u=0, \ \ u:\rone\times \rone\to {\mathbb C}. 
  \end{equation}
The well-posedness of the Cauchy problem has been studied quite recently in \cite{GM2}.  
  In particular, these models conserve the mass $M[u]=\int_\rone |u|^2$ and the Hamiltonian,  which in this case takes the form 
  $$
  \ch[u(t)]=\f{1}{2} \int_\rone |u \p_x u|^2  - \f{1}{p}\int_\rone   |u|^p dx=\ch[u_0].
  $$
  The non-linear dynamics of an equilibrium solution of \eqref{10} and \eqref{20} heavily depends on the spectral/linear stability of the equilibrium. To that end,  note that the degenerate KdV, \eqref{10}  can be written in the Hamiltonian form 
  $$
  \partial_t u=\cj \frac{\partial \mathcal{H}}{\partial u} 
  $$
  with $L^2(\mathbb{R})$ skew-symmetric  operator $\mathcal{J}=\partial_x:H^1(\mathbb{R})\subseteq L^2(\mathbb{R})\to L^2(\mathbb{R})$, while the degenerate NLS, in the form 
  $$
  \p_t z = \cj \frac{\partial \mathcal{H}}{\partial z} 
  $$
  where $z= \left(\begin{array}{c}
  	 u \\ \bar{u} 
  \end{array} 
  \right) $, $\cj=\left(\begin{array}{cc}
  	0 & 1 \\ -1 & 0
  \end{array}\right)$. 

  Of particular interest will be the solitary wave solutions of \eqref{10} and \eqref{20} respectively. That is, consider solutions of the type  $u(t,x)=\phi(x-\om t)$ in \eqref{10} and $u(t,x)=e^{i \om t} \phi(x)$ in \eqref{20}, for a real-valued\footnote{In fact, we will construct non-negative solutions $\phi>0$.}  function $\phi$, 
  we are lead to the same profile equation for $\phi$, namely 
  \begin{equation}
  \label{40} 
  -\phi \p_x(\phi \p_x \phi)+\om \phi - |\phi|^{p-2} \phi=0.
  \end{equation}
  The existence of such waves, with appropriate properties, such as smoothness, decay  at infinity etc. has 
  been considered in the literature, \cite{GM1}. These results are relatively straightforward, we present a version, which suffice for our purposes, see Proposition \ref{prop:12} below.  To set the notations, we say that a function $f:\rone\to \rone_+$ is {\it bell-shaped}  if it is even, non-negative  and monotone decreasing on $(0, \infty)$. 
  
  If we restrict our considerations to bell-shaped and decaying at $\infty$ solutions, one can say much more. This is the subject of our next first existence and uniqueness result, Proposition \ref{prop:12}. 
  \begin{proposition}(Existence and uniqueness of bell-shaped compactons)
  	\label{prop:12} 
  	
  	Let $\om>0, p>2$. Then, there exists a  compactly supported bell-shaped solution $\phi$, which satisfies the profile equation \eqref{40} and consequently, 
  	\begin{equation}
  		\label{k:129} 
  		\phi'=-\sqrt{\om-\f{2}{p} \phi^{p-2}}, 0<x<L.
  	\end{equation} 
  	Letting   $L: supp \vp=[-L,L]$ and $\phi_0:=\phi(0)$, we have the formulas 
  	\begin{equation}
  	\label{pol:10} 
  	L=L(\om,p)=\f{p^{\f{1}{p-2}}}{2^{\f{p-1}{p-2}}} \om^{\f{4-p}{2(p-2)}}\int_0^1 \f{1}{\sqrt{z-z^{\f{p}{2}}}} dz; \ \  \phi_0=\phi_0(\om,p)=\left(\f{p \om}{2}\right)^{\f{1}{p-2}}. 
  	\end{equation}
  	Conversely, suppose that $\phi$ is a bell-shaped solution $\phi$, which vanishes at $\infty$. Then, $\phi$ is necessarily of the form described in \eqref{k:129}. In particular, it is compactly supported and unique. More precisely,  for each $\om>0, p>2$, there exists an unique bell-shaped element $\phi=\phi_{\om,p}$, with the properties described above. 
  \end{proposition}
   As alluded above, our main interest is in the stability of the waves $\phi$. In order to introduce the relevant notions, we need to derive the corresponding  linearized dynamics. We develop the necessary background material in the next section.
  \subsection{Linearizations about the solitary waves and spectral stability}
  Take the ansatz  $u=\phi(x-\om t)+e^{\la t} v(x-\om t)$ and plug it in the degenerate KdV model \eqref{10}. Ignoring high order terms $O(v^2)$, we arrive at the following linearized system 
 \begin{equation}
 \label{Lin}
 \mathcal{J}\ch_+ v=\la v.
 \end{equation}
  where $\cj=\p_x$, while 
  \begin{equation}
  	\label{H}
  	\ch_+ f= -\phi\p_x^2 (\phi f) + (\om   - (\phi\phi')'  - (p-1) \phi^{p-2}) f = -\phi\p_x^2 (\phi f)  - (p-2) \phi^{p-2} f, 
  \end{equation}
  where we have used the relation $\om-(\phi\phi')'=\phi^{p-1}$,  which follows readily from \eqref{40}. Note that the relation that we just used only holds on the support of $\phi$ (i.e. on the interval $[-L,L]$).  This is the relevant form of the eigenvalue problem in the context of  the degenerate KdV model \eqref{01}. 
  
  For the NLS  equation \eqref{20},  linearizing around the standing wave $e^{i\om t}\phi_{\om},$ that is taking $u=e^{i\om t}(\Phi_{\om}+e^{\la t} v)$ and plug in into \eqref{20}. Taking \eqref{40} into account and ignoring the higher order terms $O(|v|^2)$, setting the real and imaginary part of $v$ as $(\Re v,\Im v):=(v_1,v_2)$ we have 
  \begin{equation}
  \label{412} 
   \left(\begin{array}{cc}
  0 & -1 \\ 1 & 0
  \end{array}\right) \left(\begin{array}{cc}
  \ch_+ & 0 \\ 0 & \ch_-
  \end{array}\right) \left(\begin{array}{c}
  \Re v\\ \Im v
  \end{array}\right) = \la \left(\begin{array}{c}
  \Re v\\ \Im v
  \end{array}\right)
  \end{equation}
  where $\ch_+$ is as above and   
  \begin{equation}
  \label{418} 
  	\ch_- f= -\phi\p_x^2 (\phi f) + (\om   + (\phi\phi')'  -  \phi^{p-2}) f = -\phi\p_x^2 (\phi f) + 
  	2 (\om      -  \phi^{p-2}) f,
  \end{equation}
  again by \eqref{40}. 
  A concise form of \eqref{412} is given by 
  \begin{equation}
  \label{eig2}
  \mathcal{J}\ch \vec{v}=\lambda \vec{v}
  \end{equation}
  where we have introduced 
  $$
  \cj:=\left(\begin{array}{cc}
  0 & -1 \\ 1 & 0
  \end{array}\right) , \ch:=\left(\begin{array}{cc}
  \ch_+ & 0 \\ 0 & \ch_-
  \end{array}\right).
  $$
  It is now an opportune time to discuss the domains of the linearized operators $\ch_\pm$. The operator $\ch_\pm$ may be viewed as self-adjoint extensions of the symmetric operators, associated with the quadratic form, acting on $u,v\in H^1(\rone)$, 
  $$
  q_{+}(u,v)=\dpr{(\phi u)'}{(\phi v)'}+\om \dpr{u}{v} - \dpr{((\phi \phi')'+(p-1) \phi^{p-2})u}{v}, 
  $$
  and similar for $\ch_-$. As is well-known, such extensions are generally not unique, but we allow $\ch_\pm$ t o be {\it any} such extension, for our purposes below. 
  
  The  notion of spectral stability is as follows. 
  \begin{definition}
  	\label{defi:ss} Let $\ch_+$ be any self-adjoint extension for the symmetric operator, defined with the form $q_+$. We say that the solution $\phi(x-\om t)$ of the generalized KdV problem \eqref{10} is  spectrally stable, if the eigenvalue problem \eqref{Lin} does not have non-trivial solutions $(\la, v): \Re \la>0, v\neq 0$. 
  	
  	Similarly, the solution $e^{i \om t} \phi$ of the degenerate NLS equation \eqref{20} is called spectrally stable, if the eigenvalue problem \eqref{eig2} does not have non-trivial solutions 
  	$(\la, \vec{v}): \Re \la>0, \vec{v}\neq 0$.
  \end{definition}
  Note that in this definition, we completely sidestep the important issue for local/global well-posedness of the corresponding Cauchy problems. The local aspect of the theory is discussed in the  recent paper \cite{GM2}, but the global well-posedness theory (which is more relevant as far as stability is concerned),  seems lacking at the moment. 
  
  We now state  our main results. 
   \subsection{Main results}
   We start with an existence result for the waves $\phi$, which must satisfy the ordinary differential equation \eqref{40}. In our construction of the waves, we take advantage of a variational construction, so we introduce  our main players.  More specifically, consider the following    constrained minimization problem 
   \begin{equation}
   \label{30} 
   \left\{\begin{array}{l} 
   N[u]=\int_{\rone} |u \p_x u|^2 +\om \int |u|^2\to \min  \\
   \textup{subject to} \  \int |u|^p dx=1.
   \end{array}   \right. 
   \end{equation}
   Equivalently, the same solutions are achieved via the minimization of the so-called Weinstein functional, 
   $$
   J_\om[u]:=\f{\int_{\rone} |u \p_x u|^2 +\om \int |u|^2}{\|u\|_{L^p}^2}. 
   $$
   Note that formally, $J_\om[u]$ is unconstrained, but one can see that without loss of generality, one may consider only $u: \|u\|_{L^p}=1$, which is of course \eqref{30}. 
   The following is the main result of this paper. 
   \begin{theorem}
   	\label{theo:15} 
   	Let $p>2$ and $\om>0$. Then, there exists an unique bell-shaped solution $\phi$ of the profile equation  \eqref{40}, which is a compacton (i.e. it has a compact support).  Their half-period $L$ and their amplitude are given by \eqref{pol:10}. Moreover,  such solutions are the unique constrained minimizers of the variational problem \eqref{30}.  
   	By construction,  $\phi(x-\om t)$ is a traveling wave solution of the degenerate KDV \eqref{10}, while $e^{i \om t} \phi$ is a standing wave solution of the degenerate NLS, \eqref{20}. 
   	
   	Regarding spectral stability, $\phi(x-\om t)$ and $e^{i \om t} \phi(x)$ are spectrally stable solutions (of \eqref{10} and \eqref{20} respectively) if and only if  $2<p\leq 8$. 
   \end{theorem}
\noindent {\bf Remark:} 
\begin{enumerate}
	\item Recall the issue with possible different self-adjoint extensions of the operators $\ch_\pm$. In particular,   the spectral instability results stated above,  should be interpreted as follows - for 
	every self-adjoint extension of  $\ch_\pm$, there is a non-trivial solution $(\la, v)$ of the eigenvalue problem \eqref{Lin} (\eqref{eig2} respectively). 
	\item  For the values $2<p<8$, the solution $\phi$ may be generated as a normalized wave. That is, as the minimizer of the following constrained variational problems
	$$ 
	\left\{\begin{array}{l} 
		N[u]=\f{1}{4} \int_{\rone} |u \p_x u|^2   - \f{1}{p} \int |u|^p\to \min  \\
		\textup{subject to} \  \int |u|^2 dx=\la.
	\end{array}   \right. 
	$$
	Note that this constrained variational problem also has unique solution, for each $\la>0, p>2$. Then, it can be proved that there is a one-to-one correspondence  $\om=\om(\la):\rone_+\to \rone_+$, which is actually an increasing function. 
\end{enumerate}

  \section{Preliminaries} 
 
  \subsection{Some basics} 
   We  use the standard expressions for $\|\cdot\|_{L^p(\rone)}, 1\leq p\leq \infty$ as well as the following expression for the Fourier transform and its inverse 
   $$
   \hat{f}(\xi)=\int\limits_{\rone} f(x) e^{-2\pi i x \xi} dx, \ \ \ f(x)  =\int\limits_{\rone}  \hat{f}(\xi) e^{2\pi i x  \xi} d\xi. 
   $$
  Naturally, the Laplacian may be defined through its symbol, that is $\widehat{-\De f}(\xi)=4\pi^2 |\xi|^2 \hat{f}(\xi)$. 
   Homogeneous and non-homogeneous Sobolev spaces are defined via 
   $$
   \|f\|_{\dot{W}^{s,p}}=\|(-\De)^{\f{s}{2}} f\|_{L^p},  \|f\|_{W^{s,p}}= \|(-\De)^{\f{s}{2}} f\|_{L^p}+\|f\|_{L^p}. 
   $$
   The Sobolev embedding will be useful in the sequel, so we state it here - for each $1<r<q<\infty$, 
   $$
   \|f\|_{L^q}\leq C_{r,q} \|f\|_{\dot{W}^{\f{1}{r}-\f{1}{q},r}}.
   $$
  \subsection{Rearrangements}
  In this subsection, we discuss the techniques  of rearrangements. For a measurable function $f:\rone\to \rone$, let 
  $$
  d_f(\al):=|\{x\in \rone: |f(x)|>\al\}|
  $$
  be its \emph{distribution function}. For every $\vp\in C^1(\rone)$  with $\vp(0) = 0$, 
  one has  {\it the layer cake representation formula} 
  \begin{equation}\label{150}
  \int_{\rone} \vp(|f(x)|) dx= \int_0^\infty \vp'(\al) d_f(\al)\,\mathrm{d}\al. 
  \end{equation}
  The \emph{non-increasing rearrangement}, $f^*$, of the function $f$ is the inverse function of $d_f$, provided that $\al\to d_f(\al)$ is strictly decreasing. In general, we define  $f^*:\rone_+\to\rone_+$ by
  $$
  f^*(t)=\inf\{s>0: d_f(s)\leq t\}, \qquad t \geq 0. 
  $$ 
  Then $f^*$ is a non-increasing function, and one furthermore has that $d_{f^*}(\al)=d_f(\al)$.  
  Let $f^\#(t)=f^*(2|t|)$, $t \in \rone$.  Then $f^\#$ provides us with a convenient way of characterizing bell-shapedness, namely, we say that $f$ is bell-shaped if and only if $f^\# =f$. Note that in particular, bell-shaped functions are non-negative. 
 
  \begin{lemma}(Polya-Szeg\"o inequality)
  	\label{ps} 
  	
  	For   $f\in H^1(\rone)$, we have  that $f^\# \in H^1(\rone)$ and in fact, 
  	\begin{equation}
  	\label{eq:ps} 
  	\|f'\|_{L^2(\rone)} \geq \| \p_x  f^\# \|_{L^2(\rone)}.
  	\end{equation}
  \end{lemma}
  The next result is the  Hardy's inequality.
  \begin{lemma}
  	\label{Hardy} 
  	Let $a<b$ and $f\in H^1(\rone)$, so that $f(a)=0$. Then, 
  	\begin{equation}
  	\label{har} 
  		\int_a^b \frac{|f(x)|^2}{|x-a|^2} dx\leq 	\int_a^b |f'(x)|^2 dx. 
  	\end{equation}
  \end{lemma}
  {\bf Remark:} This is a slightly more general version of the classical statement 
  $$
  \int_{-\infty}^\infty \f{|f(x)|^2}{|x|^2} dx\leq \int_{-\infty}^\infty |f'(x)|^2 dx.
  $$
   for function $f: f(0)=0$. But it is clear that from the classical version, 
   one obtains by an  approximation argument that 
   $$
    \int_{0}^b \f{|g(x)|^2}{|x|^2} dx\leq \int_{0}^b |g'(x)|^2 dx
    $$
     for every $g: g(0)=0$, $b>0$. It is then clear that the formulation \eqref{har} reduces to this form, by taking $g: g(x)=f(x+a)$. 
  \subsection{The basics of the Hamiltonian index theory} 
  \label{sec:2.3}
  We follow the setup described in \cite{LZ}, but the earlier versions of these results, \cite{KKS, KKS2, KP, Pel}  provided much impetus in the developments of these methods. 
  
  Consider the Hamiltonian eigenvalue problem of the form
  \begin{equation}
  \label{gHam}
  \mathcal{J}\mathcal{H}g=\la g, u\in H
  \end{equation}
  where $H$ is a Hilbert space,  
  $\mathcal{J}:\mathcal{D}(\mathcal{J})\subset H^*\to H$, 
  and $$\mathcal{H}:H\to H^*$$ We will give a brief introduction of the analysis of the number of unstable eigenvalues of \eqref{gHam}. Assume that $\langle \mathcal{H}u,v\rangle$ is bounded symmetric bilinear form on $H\times H$, which gives rise to a self-adjoint operator $(\ch, D(\ch))$ Moreover, there exist a decomposition of $H$ in $\ch$ invariant 
  $$
  H=H_-\oplus ker\mathcal{H}\oplus H_+
  $$
  where  $H_-, H_+$  are the negative and positive subspaces respectively. More precisely, upon introducing the self-adjoint projections $P_-=\chi_{(-\infty, 0)}(\ch), P_+=\chi_{(0, +\infty)}(\ch)$, we take  $H_-=P_-[H], H_+=P_+[H]$. Assume in addition that the Morse index of $\ch$, that is $n(\ch)=dim(H_-)<\infty$, while there exists $\de>0$, so that $\dpr{\ch u}{u}\geq \de \|u\|^2$, for all $u\in D(\ch)\cap H_+$. For any $\la\in \si_{p.p.}(\cj\ch)$, introduce the generalized eigenspace 
  $$
  E_{gen}^\la:=\cup_{j=1}^\infty E_j^\la, E_\la^j=\{f\in H: (\mathcal{JH}-\la I)^jf=0\}
  $$ 
  Assume also that the dual space $H^*$ satisfies 
  $$
  \{\langle g\in H^*:\langle g,u\rangle=0,\forall u\in H_-\oplus H_+ \rangle\}\subset D(\mathcal{H}). 
  $$
  Further since, $Ker(\ch)\subset E_{gen}^0$, decompose  $E_{gen}^0=ker(\mathcal{H})\oplus E_0$, where $E_0$ is finite dimensional with basis say $\{\psi_i\}_{i=1}^{n}\subset \mathcal{D}(\mathcal{H})$.  Let $D$ be the  matrix with entries 
  $$
  D_{ij}=\langle\mathcal{H}\psi_i,\psi_j\rangle. 
  $$
  Next, we need to introduce the notion of negative Krein signature  of a purely imaginary  eigenvalue. More specifically, for $i\mu\in \si_{p.p.}(\cj\ch), \mu>0$, consider $E_{i \mu}:=Ker(\cj \ch-i \mu), P_{i \mu}: H\to E_{i \mu}$ 
  and $k_i^{i \mu}=n(\ch|_{E_{i \mu}})=n(P_{i\mu} \ch P_{i \mu})$. Finally, we define the total Krein signature 
  $$
  k_i^{\leq 0}=\sum\limits_{\mu\neq 0: i\mu\in \si_{p.p.}(\cj\ch)} k_i^{i \mu}.
  $$
   In the most common case, when $i\mu$ is a simple eigenvalue, with say an eigenvector $\psi_\mu$, $i \mu$ is of negative Krein signature  exactly when the real quantity $\dpr{\ch \psi_\mu}{\psi_\mu}<0$. In such a case $k_i^{i \mu}=1$. 
  
  By \cite{LZ} see also  \cite{KKS2}\cite{KKS} we have the following Hamiltonian-Krein  index formula
  \begin{equation}
  \label{HKIF}
  k_{Ham.}:=k_r+2k_c+2k_i^{\leq 0}=n(\mathcal{H})-n(D)
  \end{equation}
  where $k_r$ is the number of real positive eigenvalues of $\mathcal{JH}$ (counted with their multiplicities),  $k_c$ is the number of eigenvalues of  $\mathcal{JH}$   with positive real and imaginary part, and lastly $k_i^{\leq 0}$ is the total Krein signature introduced above. Note that $k_{Ham}=0$ implies spectral stability for the model \eqref{gHam}, as $k_{Ham}$ counts all the instabilities, since $k_{Ham}\geq k_r+k_c$.  
  
  In the particular case $n(\cl)=1$, the formula \eqref{HKIF} reduces the situation to two cases, namely $k_{Ham}=1$, if $n(D)=0$ and $k_{Ham}=0$, if $n(D)=1$. We have already discussed that $k_{Ham}=0$ is a case of stability. If however $k_{Ham.}=1$, by parity considerations, 
  it follows that $k_r=1$, while $k_c=k_i=0$. In an y event, this implies that the system has a real unstable growing mode. Thus, we have the following useful corollary of \eqref{HKIF}. 
  \begin{corollary}
  	\label{cor:12} 
  	If $n(\cl)=1$, the the eigenvalue problem \eqref{gHam} is spectrally stable if  $n(D)=1$, and it has exactly one real unstable mode, if $n(D)=0$. 
  \end{corollary}

  \section{Existence of the compacton waves} 
  We start with the proof of Proposition \ref{prop:12}. 
  \subsection{Proof of Proposition \ref{prop:12}}
  We solve the ODE \eqref{40}. As long as $\phi\neq 0$, say on an interval $[-L,L]$, we may divide by $\phi$, which then leads us to the ODE 
 \begin{equation}
 \label{41} 
 -\f{1}{2} \p_x^2 (\phi^2) + \om - |\phi|^{p-2}=0.
 \end{equation}
 Denoting $Q:=\phi^2\geq 0$, this is equivalent to 
 \begin{equation}
 \label{42} 
 -\f{1}{2} Q'' + \om - Q^{\f{p}{2}-1}=0.
 \end{equation}
 Multiplying the equation by $Q'$ and integrating on the interval $[-L,L]$, and imposing  that $Q'(-L)=Q'(L)=Q(L)=Q(-L)=0$, we reduce the order of the ODE, namely to 
 \begin{equation}
 \label{43} 
 (Q')^2 = 4 \om Q - \f{8}{p}  Q^{\f{p}{2}}. 
 \end{equation}
 If we further require that $Q$ is decreasing on $(0,L)$, i.e. $Q'(r)<0, r>0$, we can finally setup the ODE 
 \begin{equation}
 \label{44} 
 Q' (r)= -2\sqrt{ \om Q(r) - \f{2}{p}  Q^{\f{p}{2}}(r)}, 0<r<L.
 \end{equation}
Note that this immediately implies \eqref{k:129}. 
Clearly, \eqref{44} has an unique solution of the required form, namely with $Q(0)=Q_0=\left(\f{p \om}{2}\right)^{\f{2}{p-2}}, Q'(0)=0$, if we select 
\begin{equation}
	\label{l:10} 
	L=\int_0^{Q_0} \f{1}{2\sqrt{\om Q- \f{2}{p} Q^{\f{p}{2}}}} dQ = \f{p^{\f{1}{p-2}}}{2^{\f{p-1}{p-2}}} \om^{\f{4-p}{2(p-2)}}\int_0^1 \f{1}{\sqrt{z-z^{\f{p}{2}}}} dz.
\end{equation}
Clearly, such solution is bell-shaped as it satisfies \eqref{44}. 

Conversely, assume that $\phi$ is a bell-shaped solution of  \eqref{40}, which vanishes at $\pm \infty$. Then, $Q:=\phi^2$ will satisfy \eqref{42} as argued above, and it will also vanish at $\pm \infty$. We observe now that the support of $Q$ may not be infinite, since then, it must be that 
$$
+\infty=\int_0^{Q(0)} \f{dQ}{\sqrt{\om Q-\f{2}{p}Q^{\f{p}{2}}}}
,$$
which is clearly false as the integral is convergent, due to the mild singularity at both $0$ and $Q(0)$. Thus, $Q$ (and subsequently $\phi$) is supported on a finite interval $[-L,L]$. Then, it becomes clear that since $Q'(L)=0$, it is the case that $Q(0)=Q_0=\left(\f{p \om}{2}\right)^{\f{2}{p-2}}$ and the solution is unique from \eqref{44}. In particular, $L$ is given by the formula displayed in \eqref{l:10}. 

  \subsection{An alternative  variational problem} 
  Looking at the form of the functional $N[u]$,  it is pretty standard to replace $u=\sqrt{v}$, especially since we are looking for positive solutions of \eqref{40}. Specifically, we shall consider 
  \begin{equation}
  \label{34} 
  \left\{\begin{array}{l} 
  N_0[v]=N[\sqrt{v}]=\f{1}{4} \int_{\rone} |v'|^2 +\om \int |v| \to \min  \\
  \textup{subject to} \  \int |v|^{p/2} dx=1.
  \end{array}   \right. 
  \end{equation}
  Clearly, $N_0[v]\geq 0$, so we introduce 
  $$
  m(\om):=\inf_{v\in \cs} J_\om(u)= \inf\limits_{\int |u|^{p} dx=1} N[u] = \inf\limits_{\int |v|^{p/2} dx=1} N_0[v].
  $$
  We now show that \eqref{34} has a solution. 
  \begin{proposition}
  	\label{prop:10} 
  	Let $p>2, \om>0$. Then, the constrained minimization problem \eqref{34} has a bell-shaped solution $\vp=\vp^\#$. 
  \end{proposition}
  \begin{proof}
  	By the Szeg\"o inequality, $ \int_{\rone} |v'|^2 \geq  \int_{\rone} |\p_x v^\#|^2$, while 
  	$\int |v|=\int |v^\#|, \int |v|^{p/2}=\int |v^\#|^{p/2}$, whence it is clear that it suffices to restrict the problem \eqref{34}to bell-shaped entries $v$. Take a minimizing sequence, $v_n: \|v_n\|_{L^{p/2}}=1$, $v_n\in H^1\cap L^1$, $\lim_n \left(\f{1}{4} \int_{\rone} |v_n'|^2 +\om \int v_n\right) = m(\om)$. It follows that $\sup_n \|v_n'\|_{L^2}< 4m(\om)$, $\sup_n \int v_n \leq m(\om)$. By the bell-shapedness, we conclude the point-wise decay $|v_n(x)|\leq \f{C}{1+|x|}$. Thus, by Kolmogorov-Rellich criteria, the set \\ 
  	$\{v_n: \sup_n \|v_n'\|_{L^2}<\infty, |v_n(x)|\leq \f{C}{1+|x|}\}$ is pre-compact in $L^{p/2}(\rone)$. Without loss of generality, we can assume that $\lim_n \|v_n-\vp\|_{L^{p/2}}=0$ and $v_n'\to \vp'$ weakly. Similarly, fixing a compact subset $K\subset \rone$, the set $\{v_n:   \sup_n \|v_n'\|_{L^2}<\infty, \|v_n\|_{L^{p/2}}=1\}$ is a pre-compact subset in any $C^{\al}(K), \al<\f{1}{2}$. In particular, we can without loss of generality assume in addition that $v_n$ converges uniformly to $\vp$ on the compact subsets of $\rone$. By Fatou's lemma, $\liminf_n \int v_n\geq \int \vp$. Finally, by the lower semi-continuity of the  $L^2$ norm, with respect to the weak topology, $\liminf_n \int (\p_x v_n)^2 \geq \int (\vp')^2$. Putting it all together, we have that $\int \vp^{p/2}=1$, while 
  	$$
  	N_0[\vp]\leq \liminf_n N_0[v_n]=m(\om).
  	$$
  	It follows that $N_0[\vp]=m(\om)$, whence $\vp$ is indeed a constrained minimizer of \eqref{34}. In addition, $\vp$ is clearly bell-shaped (as limit of bell-shaped functions). 

  \end{proof}
  We will eventually establish that the minimizers $\vp$ of \eqref{34} are unique (up to translations) and they also have compact support, whence the reference to ``compactons''. We would like the reader to keep this in mind, as this is somewhat non-standard situation, which develops herein. For this, we need to derive the Euler-Lagrange relation for \eqref{34}.
    \begin{proposition}
    	\label{prop:22} 
    	The solution $\vp$ of \eqref{34} is compactly supported function, which  satisfies the Euler-Lagrange equation 
    	\begin{equation}
    	\label{61} 
    	-\frac12\vp''+\omega-c(\omega,p)\vp^{\frac p2-1}=0, -L<x<L
    	\end{equation}
    	where $c=c(\om,p)$ are explicit. In fact, there is the formula
    	\begin{eqnarray}
    	\label{300} 
    	c(\om,p) = p^{\f{3}{p+1}}  \om^{\f{p+4}{2(p+1)}} \left(2 \int_0^1 \sqrt{z-z^{\f{p}{2}}} dz+ \int_0^1 \f{z}{\sqrt{z-z^{\f{p}{2}}}} dz\right)^{\f{p-2}{p+1}}. 
    	\end{eqnarray}
    	Finally, there are the following formulas for the behavior of $\vp$ at $\pm L$
    	\begin{equation}
    	\label{712} 
    	 \vp(x)= \om (L-x)^2+O((L-x)^3),  \ \  \vp(x)= \om (x+L)^2+O((x+L)^3).
    	\end{equation}
    \end{proposition}
Before we proceed with the proof of Proposition \ref{prop:22}, we would like to make some important remarks. 
\begin{enumerate}
	\item Similar to \eqref{300}, one may compute various quantities involving norms of $\vp$, as well as the 
	half-period $L$ etc. This is despite the lack of explicit formulas for the function $\vp$. 
	\item It is pretty clear that once $c$ is given by a formula like \eqref{300}, the uniqueness results from Proposition \ref{prop:12} apply here as well. In particular, the variational problem \eqref{34} 
	has an unique bell-shaped solution $\vp$. 
\end{enumerate}
    \begin{proof}(Proposition \ref{prop:22})
    	
    	 Let  $\eps: |\eps|<<1$ and consider a test function $h$, so that  it vanishes  outside the support of $\vp$.  Let us look at  a perturbation $\vp+\eps h$. Note that due to the constraint in \eqref{34} and the support property of $h$, namely $supp \ h\subset supp \ \vp$, we have  the formula 
    	 $$
    	 \|\vp+\eps h\|_{L^{\f{p}{2}}} =1 + \eps \dpr{\vp^{\f{p}{2}-1}}{h}+O(\eps^2).
    	 $$
    	   From the minimization property   of $\vp>0$, we must have that the scalar function 
    	 \begin{eqnarray*}
    	 g(\eps):&=& N_0\left[\f{\vp+\eps h}{\|\vp+\eps h\|_{L^{\frac{p}{2}}}}\right]=\frac14\int_{\mathbb{R}}\left(\frac{\vp'+\epsilon h'}{\|\vp+\epsilon h\|_{L^{\frac{p}{2}}}}\right)^2dx+ \om \int_{\rone} \f{(\vp+\eps h)}{\|\vp+\eps h\|_{L^{\f{p}{2}}}} dx  = \\
    	 &=&  \frac14\int_{\rone} (\vp'+\epsilon h')^2 dx (1-2\eps\dpr{\vp^{\f{p}{2}-1}}{h} ) +\om \int_{\rone} (\vp+\eps h)dx (1-\eps\dpr{\vp^{\f{p}{2}-1}}{h} )+O(\eps^2)\\
    	 &=& g(0)+ \eps \left( \f{1}{2} \dpr{-\vp''}{h}-\dpr{\vp^{\f{p}{2}-1}}{h} (\f{1}{2} 
    	 \int_{\rone} (\vp')^2 dx+\om \int \vp) +\om \dpr{h}{1}\right)+O(\eps^2). 
    	 \end{eqnarray*}
    \end{proof}
    achieves its minimum at zero. It follows that $g'(0)=0$, whence for all test functions $h$
    $$
    \dpr{-\f{1}{2}\vp''+\om - c \vp^{\f{p}{2}-1}}{h}=0,
    $$
    where we have denoted $c=c(\om,p):=\frac 12\int (\vp')^2 +\om \int \vp$. 
    It follows that on the support of $\vp$, the Euler-Lagrange equation \eqref{61} holds true. As $\vp$ is bell-shaped, this must be an interval of the form $[-L,L]$ for some $L>0$ or $\rone$. 
    
   We show now that such minimizer $\vp$ must be compactly supported. 
   Indeed, suppose that $supp \ \vp$ is $\rone$. Multiplying \eqref{61} by $\vp'$ and integrating once (and taking into account that $\vp, \vp'$ vanish at $\pm \infty$ and $\vp'(x)<0, x>0$, due to bell-shapedness),  we obtain the relation 
   \begin{equation}
   \label{62} 
   	\vp'(x)= -2\sqrt{\om \vp(x) - \f{2 c }{p}\vp^{\f{p}{2}}(x)}, x>0.
   \end{equation}
    Now, clearly as the function $\vp$ achieves its maximum at zero, we have that $\vp_0$ is a   zero of the function $ z\to \om  - \f{2 c }{p}z^{\f{p}{2}-1}$, so $\vp_0=\vp(0)$ satisfies 
    \begin{equation}
    \label{64} 
    \vp_0^{\f{p}{2}-1}=\f{p \om }{2c(\om,p) }.
    \end{equation}
    If the relation \eqref{62} holds for a function $\vp$ supported on $\rone$, it must be that  
    $$
    \infty= \int_0^{\vp_0} \f{d\vp}{2\sqrt{\om \vp - \f{2 c}{p}\vp^{\f{p}{2}}}}. 
    $$
    This  is however  clearly false, as the integral in question is convergent, due to the mild singularities at $0, \vp_0$. So, $supp \ \vp=[-L,L]$, and in fact, using the relation \eqref{64} yields 
    $$
    L=\int_0^{\vp_0} \f{d\vp}{2\sqrt{\om \vp - \f{2 c(\om)}{p}\vp^{\f{p}{2}}}} = \f{\sqrt{\vp_0}}{2\sqrt{\om}} \int_0^1 \f{dz}{\sqrt{z-z^{p/2}}}.
    $$
   We now  compute explicitly $c(\om,p)$. We have 
   \begin{eqnarray*}
   c(\om,p) &=& \frac 12\int (\vp')^2 +\om \int \vp = \int_0^L (\vp')^2 +2 \om \int_0^L \vp = \\
   &=&  2\int_0^{\vp_0} \sqrt{\om \vp - \f{2 c}{p}\vp^{\f{p}{2}}} d\vp+\om \int_0^{\vp_0} \f{\vp}{\sqrt{\om \vp - \f{2 c}{p}\vp^{\f{p}{2}}}} d\vp=\\
   &=& \vp_0^{\f{3}{2}} \sqrt{\om} \left(2 \int_0^1 \sqrt{z-z^{p/2}} dz+ \int_0^1 \f{z}{\sqrt{z-z^{p/2}}} dz\right).
   \end{eqnarray*}
    Together with \eqref{64}, this yields a system of two relations for the unknowns $\vp_0(\om,p), c(\om,p)$, which results in the formula \eqref{300}. 
   
    Finally, we discuss the behavior of $\vp$ in a proximity of the endpoints $\pm L$. From bell-shapedness and \eqref{61}, we obtain that $\vp(\pm L)=\vp'(\pm L)=0$, whereas $\vp''(\pm L)=2\om>0$. Thus, \eqref{712} holds true.

    \subsection{Spectral theory  for the linearized operator $\cl_+$}
    
    Clearly, by the equivalence  between the constrained minimization problems \eqref{30} and \eqref{34}, we have that \eqref{30} has a solution $\Phi:=\sqrt{\vp}$, which is also bell-shaped. We establish further properties of $\Phi$.
    \begin{proposition}
    	\label{prop:23} 
    	The solution $\Phi$ of \eqref{30}, satisfies the Euler-Lagrange equation 
    	\begin{equation}
    	\label{60} 
    	-\Phi\p_x(\Phi \Phi')+\om \Phi - c(\om) \Phi^{p-1}=0, \ \ -L<x<L.
    	\end{equation}
    	In addition, consider the symmetric operator 
    	\begin{equation}
    	\label{L+}
    	\cl_+f =  -\Phi\p_x^2 (\Phi f) +\om f  - (\Phi\Phi')' f - (p-1) c(\om) \Phi^{p-2} f, 
    	\end{equation}
    	with a base Hilbert space $L^2(\rone)$.  Any self-adjoint extension of $\cl_+$ (also denoted by $\cl_+$),
    	has the property $\cl_+|_{\{\Phi^{p-1}\}^\perp}\geq 0$. 
    	In particular, (any self-adjoint extension of) $\cl_+$ has at most one negative eigenvalue. 
    \end{proposition}
   {\bf Remark:}   The operator $\cl_+$ is associated with the quadratic form 
   $$
   q(u,v)=\dpr{\p_x(\Phi u)}{\p_x(\Phi v)} +\om \int_{\rone} u(x) \bar{v}(x) dx - (p-1) c(\om, p) \dpr{ \Phi^{p-2} u}{v},
   $$
   acting on functions $u,v\in H^1(\rone)$. 
   
   By the relation \eqref{60}, we have that $-(\Phi\Phi')'+\om=c(\om)\Phi^{p-2}$,   and so we can rewrite the linearized operator and the corresponding quadratic form, on functions $f: supp f\subset (-L,L)$, as follows 
   \begin{eqnarray}
   \label{88} 
   \cl_+ f=-\Phi\p_x^2 (\Phi f)  - (p-2)c(\om) \Phi^{p-2} f \\
   \label{89} 
   q(u,v)=\dpr{\p_x(\Phi u)}{\p_x(\Phi v)} - (p-2) c(\om, p) \dpr{ \Phi^{p-2} u}{v}. 
   \end{eqnarray}
   \begin{proof}
   	For $\eps: |\eps|<<1$ and a test function $h: supp \ h\subset (-L,L)$, consider a perturbation $\Phi+\eps h$. According to the minimization property of $\Phi$, we must have that the scalar function 
   	$$
   	f(\eps):=N\left[\f{\Phi+\eps h}{\|\Phi+\eps h\|_{L^p}}\right]=\f{1}{4 \|\Phi+\eps h\|_{L^p}^4} \int (2\Phi \Phi'+2\eps (\Phi h)'+\eps^2\p_x(h^2))^2 + \f{\om}{\|\Phi+\eps h\|_{L^p}^2}\int (\Phi+\eps h)^2.
   	$$
   	has a minimum at $\eps=0$. As a  consequences of the minimization property, $f'(0)=0$, while $f''(0)\geq 0$. It is however easier to work with expansions in powers of $\eps$, instead of differentiating directly in $\eps$.  We have 
   	$$
   	\|\Phi+\eps h\|_{L^p}^p=1+p \eps \dpr{\Phi^{p-1}}{h} + \f{p(p-1)}{2} \eps^2 \dpr{\Phi^{p-2}}{h^2}+O(\eps^3).
   	$$
   	To this end, let us do the first order expansion. 
   	\begin{eqnarray*}
   		f(\eps) &=&  \int (\Phi \Phi'+\eps (\Phi h)')^2 (1-4 \eps \dpr{\Phi^{p-1}}{h})+
   		\om \int (\Phi^2+2\eps \Phi h) (1-2\eps \dpr{\Phi^{p-1}}{h})+O(\eps^2)=\\
   		&=& f(0)+ 2\eps\left(-\int\Phi(\Phi\Phi')h+\om\dpr{\Phi}{h} -\dpr{\Phi^{p-1}}{h} 
   		\left(2\int (\Phi\Phi')^2 +\om \int \Phi^2\right) \right)+O(\eps^2).
   	\end{eqnarray*}
   	Hence  we have the Euler-Lagrange equation \eqref{60}  in weak sense, with  $c(\om)=2\int (\Phi\Phi')^2 +\om \int \Phi^2$, which is the same coefficient that we have encountered in Proposition \ref{prop:22}.

   	Next, we deal with the second order condition, namely $f''(0)\geq 0$. To simplify matters, take $h\perp \Phi^{p-1}$, that is $ \dpr{\Phi^{p-1}}{h}=0$.  Looking at the next order, that is the terms containing $\eps^2$. We obtain 
   	\begin{eqnarray*}
   		\f{f''(0)}{2}=-(p-1)\dpr{\Phi^{p-2}}{h^2} \left(2  \int (\Phi \Phi')^2   +\om \int \Phi^2  \right) + \dpr{(\Phi h)'}{(\Phi h)'} +\int \Phi \Phi' \p_x (h^2)+\om \int h^2.
   	\end{eqnarray*}
   	It is now clear, after integration by parts, that one can write the previous identity in the form 
   	$$
   	\f{f''(0)}{2}=\dpr{\cl_+ h}{h}.
   	$$
   	As  this is valid for all $h: \dpr{\Phi^{p-1}}{h}=0$ and $f''(0)\geq 0$,  we conclude that  $\cl_+|_{\{\Phi^{p-1}\}^\perp}\geq 0.$
   \end{proof}
    From Proposition \ref{prop:23} and the minimax formula for the eigenvalues of a self-adjoint operator, we conclude that $n(\cl_+)\leq 1$. Our next task is to characterize  the rest of the  spectrum of $\cl_+$. We will need the relation between $\Phi_x^2$ and $\Phi\Phi_{xx}$ with $\Phi^{p-2}.$ Multiplying \eqref{60} by $\Phi'$ and integrating, taking into account that  $\Phi(\pm L)=0$,  we obtain 
    \begin{equation}
    \label{key}
    (\Phi')^2=\omega-\frac{2c(\om)}{p}\Phi^{p-2}.
    \end{equation}
    Another very important relation, which is  obtained from \eqref{60} and \eqref{key} is 
    \begin{equation}
    \label{key2}
    \Phi\Phi''=\frac{2-p}{p}c(\om)\Phi^{p-2}.
    \end{equation}
    We now state a technical result, which connects the degenerate operator $\cl_+$, posed on domain contained in $L^2[-L,L]$ to a standard Schr\"odinger operator, defined herein 
    \begin{equation}
    \label{361} 
     L_+:=-\p_t^2+\f{\om}{4} - c(\om)\f{2p^2-5p+3}{2p} \Phi^{p-2}(x(t)).
    \end{equation}
    To this end, we borrow an important idea  from \cite{GM1}. Namely,  a change of variables is introduced, which transform the degenerate operator $\cl_+$, acting on $D(\cl_+)\subset L^2[-L,L]$ into a standard Schr\"odinger operator $L_+$, with exponentially decaying potential, acting on a subspace of $L^2(\rone)$. 
    \begin{lemma}
    	\label{le:op}
    	The transformation 
    		\begin{equation}
    		\label{110} 
    		t(x)=\int_0^x\frac{dy}{\Phi(y)}:(-L,L)\to \rone
    		\end{equation}
    		is a one-to-one mapping from $(-L,L)$ to $\rone$. The inverse function $x:\rone\to (-L,L)$ satisfies 
    		$$
    		\lim_{t\to \pm \infty} x(t)=\pm L
    		$$
    		 and moreover, it approaches the limits with exponential rates. More specifically,  for every $\eps>0$, there exists $C_\eps$, so that  
    		 \begin{eqnarray}
    		 \label{152}
    		 & & L-C_\eps e^{(-\sqrt{\om}+\eps)t}<x(t)<L, \ \ t>0 \\
    		 \label{153} 
    		 & & -L<x(t)<-L+C_\eps e^{(\sqrt{\om}-\eps)t}, \ \ t<0.
    		 \end{eqnarray}
    	For the Schr\"odinger operator 	$L_+$, defined in \eqref{361}, 
    	 there is the relation 
    		 \begin{equation}
    		 \label{210} 
    		 \cl_+ f =\f{L_+ (\sqrt{\Phi} f)}{\sqrt{\Phi}}.
    		 \end{equation}
    \end{lemma}
    {\bf Remark:}  
   Note that due to the asymptotics \eqref{152} and \eqref{153}, and \eqref{712},  the  potential \\ $W(t)=c(\om)\f{2p^2-5p+3}{2p} \Phi^{p-2}(x(t))$ obeys the exponential decay estimate 
    	$$
    	0<W(t)\leq C_\eps e^{-((p-2)\sqrt{\om}-\eps) |t|}, t\in \rone 
    	$$
    	for every $\eps>0$ and some $C_\eps$. 
    \begin{proof}
    	By \eqref{712}, we have that $|\Phi(y)|\sim |L-y|, y\sim L$ and $|\Phi(y)|\sim |L+y|, 
    	y\sim -L$, we have that the integral in \eqref{110} converges to $-\infty$ as $x\to -L$, while it approaches $\infty$, as $x\to L$. That is, the transformation \eqref{110} provides a one-to-one homeomorphic map 
    	$t:   (-L,L) \to (-\infty, \infty)$. 	Its inverse, which will be frequently used, is  denoted $x(t)$. We need the asymptotic behavior of $x(t)$, which we analyze next. A simple L'Hospital calculation shows that 
    	$$
    	\lim_{x\to L-} \f{\ln(L-x)}{t(x)}=- \lim_{x\to L-}  \f{\Phi(x)}{L-x}=-\sqrt{\om},
    	$$
    	where we have used the asymptotic \eqref{712} and similar at $x=-L$. In short, we obtain that for every $\eps>0$, we have for every $\eps>0$, \eqref{152} and \eqref{153} hold true.

    	For the formula \eqref{210}, note first that  have the relation $\p_t = \Phi \p_x$. We  compute $ \Phi^{\f{1}{2}} \cl_+[f\Phi^{-\f{1}{2}}]$. By direct calculations, 
    	\begin{eqnarray*}
    		-\Phi^{\f{3}{2}}\p_x^2[\sqrt{\Phi} f] &=&   -\Phi^2 \p_x^2 f - \Phi \Phi_x \p_x f - \f{2\Phi''\Phi-(\Phi')^2}{4} f=\\
    		&=& -\Phi^2 \p_x^2 f - \Phi \Phi_x \p_x f  +\f{\om}{4} f - \f{c(\om)(3-p)}{2p} \Phi^{p-2} f,
    	\end{eqnarray*}
    	where we have used the relations \eqref{key} and \eqref{key2}. 
    	It follows that 
    	$$
    	 \Phi^{\f{1}{2}} \cl_+[f\Phi^{-\f{1}{2}}] = -\Phi^2 \p_x^2 f - \Phi \Phi_x \p_x f  +\f{\om}{4} f - c(\om) \Phi^{p-2} \f{2p^2-5 p+3}{2p}  f.
    	$$
    	On the other hand, note that 
    	$$
    	-\p_t^2 f=-\p_t (\Phi \p_x f)=-\Phi(\p_x(\Phi \p_x f))= - \Phi^2 \p_x^2 f-\Phi \Phi_x \p_x f.$$
    	It follows that $L_+ f =  \Phi^{\f{1}{2}} \cl_+[f\Phi^{-\f{1}{2}}] $, or alternatively, \eqref{210}. 
    \end{proof}
    \begin{proposition}
    	\label{prop:45} Let  $\cl_+$ be any self-adjoint extension of the symmetric  operator introduced in \eqref{L+}.  Then, $\cl_+$ has exactly one negative eigenvalue, say $-\zeta^2$, a simple eigenvalue at zero, with $Ker(\cl_+)=span[\Phi']$, while the rest of the spectrum is away from zero. That is, there exists $\de>0$, so that 
    	\begin{equation}
    	\label{215} 
    		spec(\cl_+)\setminus \{-\zeta^2, 0\}\subset [\de, +\infty).
    	\end{equation}
    \end{proposition}
  \begin{proof}
  	Direct inspection shows that 
  	\begin{eqnarray*}
  	\cl_+\Phi' &=& -\Phi\p_x^2(\Phi \Phi')+\om \Phi'-(\Phi \Phi')'\Phi'- (p-1) c(\om) \Phi^{p-2} \Phi'=\\
  	&=& -\Phi\p_x^2(\Phi \Phi') -(\Phi \Phi')'\Phi' +(\om \Phi - c(\om)\Phi^{p-1})'=\\
  	&=& -\Phi\p_x^2(\Phi \Phi') -(\Phi \Phi')'\Phi'+(\Phi\p_x(\Phi\Phi'))'=0,
  	\end{eqnarray*}
  	whence $\Phi'\in Ker(\cl_+)$. This is of course a consequence of the translational invariance of the profile equation \eqref{60}. Also by a direct inspection, and using the equation \eqref{60}, we conclude 
  	$$
  	\cl_+\Phi=-2\om \Phi-(p-4) c(\om) \Phi^{p-1}.
  	$$
  	Taking dot product with $\Phi$, we obtain, 
  	\begin{equation}
  		\label{l:20} 
  			\dpr{\cl_+ \Phi}{\Phi}=-2\om \int\Phi^2 - (p-4) c(\om).
  	\end{equation}
  
  	We will now show that such quantity is negative, for all $p>2$,  which would establish that $\cl_+$ has exactly one negative eigenvalue. One could use the arguments of Proposition \ref{prop:22} to evaluate the quantity $	\dpr{\cl_+ \Phi}{\Phi}$ explicitly in terms of $p,\om$. Instead, we provide an easier roundabout argument, which shows $	\dpr{\cl_+ \Phi}{\Phi}<0$. 
  	To this end, recall that we have already established the relation 
  	$$
  	c(\om) = \frac 12\int (\vp')^2 +\om \int \vp= 2\int (\Phi \Phi')^2 + \om \int \Phi^2.
  	$$
  	We now obtain another identity based on taking dot product of \eqref{60} with $x\Phi'$. Due to $\int \Phi^pdx=1$, we have 
  	\begin{eqnarray*}
  		c(\om)=\f{p \om}{2} \int \Phi^2 - \f{p}{2} \int (\Phi \Phi')^2.
  	\end{eqnarray*}
  	Combining this with \eqref{60}, we establish the relationship 
  	$$
  	\int (\Phi \Phi')^2 = \f{p-2}{p+4} \om \int \Phi^2 dx,
  	$$
  	which in turn implies 
  	$$
  	c(\om)=\f{3p\om}{p+4}\int \Phi^2.
  	$$
  	Substituting this in \eqref{l:20} results in the formula 
  	$$
  		\dpr{\cl_+ \Phi}{\Phi}=-\om \f{(3p^2-10p+8)}{p+4}\int \Phi^2 dx<0, 
  	$$
  	for $p>2, \om>0$. 
  	Thus, $n(\cl_+)=1$. 
  	
  	Next, we establish that the eigenvalue at zero is simple and there is a gap between the zero and the non-negative portion of the spectrum. Recall that we work with arbitrary self-adjoint extension of $\cl_+: D(\cl_+)\subset L^2[-L,L]$, given by the quadratic form $q$. Since $q(u,u)\geq 0$ for all $u\perp \Phi^{p-1}$, it follows that $\cl_+$ has only one simple negative eigenvalue, say  $-\zeta^2$, and $\si(\cl_+)\setminus \{-\zeta^2\}\subset [0, \infty)$. For any  $\Om$, a relatively open subset in $[0, \infty)$, denote the spectral projection $P_\Om:=\chi_\Om (\cl_+)$. Supposing for a contradiction that \eqref{215} fails, we select a sequence $f_n$, with $P_{[0, \f{1}{n})} f_n=f_n, \|f_n\|_{L^2}=1, f_n\perp \Phi'$. It follows that 
  	\begin{equation}
  	\label{240} 
  		\dpr{\cl_+ f_n}{f_n}=q(f_n, f_n)=\|(\Phi f_n)'\|_{L^2[-L,L]}^2 - 
  		(p-2) c(\om) \int_{-L}^L \Phi^{p-2}  f_n^2 dx\to 0.
  	\end{equation}
  	Denote $g_n:=\Phi f_n$. Clearly, $\|g_n\|_{L^2}\leq \|\Phi\|_{L^\infty}$, while 
  	$$
  	\limsup_n \|g_n'\|_{L^2}=\limsup_n [q(f_n,f_n)+(p-2) c(\om) \int_{-L}^L \Phi^{p-2}  f_n^2] \leq C_p \|\Phi\|_{L^\infty}^{p-2},
  	$$
  	 whence  $\sup_n \|g_n\|_{H^1}<\infty$. 
  	 
  	 From the various convergences and the weak compactness of bounded sets in $L^2[-L,L]$, and the compactness of the embedding $H^1[-L,L]$ into $L^2[-L,L]$, it follows that there exist a subsequence (denoted the same), so that the weak convergences 
  	$f_n\rightharpoonup_{L^2} U$, $g_n\rightharpoonup_{H^1} g$ hold, as well as the strong convergence  $\lim_n \|g_n-g\|_{L^2[-L,L]}=0$. In particular, one can see that for every $\de>0$, $\lim_n \|g_n-g\|_{L^2[-L+\de,L-\de]}=0$. Moreover, since $\Phi$ does not vanish on $[-L+\de, L-\de]$, we have that $g=\Phi U$ and 
  	$\lim_n \|f_n-U\|_{L^2[-L+\de, L-\de]}=0$, for every $\de>0$. This is now enough to conclude that 
  	\begin{equation}
  	\label{220} 
  	\lim_n \int_{-L}^L \Phi^{p-2}  f_n^2 dx = \int_{-L}^L \Phi^{p-2}  U^2 dx
  	\end{equation}
  Indeed, since $\Phi(x)\leq C(L+x), -L<x$ and $\Phi(x)\leq C(L-x), x<$, we obtain 
  $$
  | \int_{[-L, -L+\de)\cup (L-\de, L)} \Phi^{p-2}  f_n^2 dx |\leq C \de^{p-2} \|f_n\|_{L^2}^2= C \de^{p-2}, 
  $$
  and similarly for the integrals  
  $|\int_{[-L, -L+\de)\cup (L-\de, L)} \Phi^{p-2}  U^2 dx|$.  On the other hand, on the interval  $[-L+\de, L-\de]$, we use the convergence $\lim_n \|f_n-U\|_{L^2[-L+\de, L-\de]}=0$, to estimate 
  $$
  | \int_{-L+\de}^{L-\de} \Phi^{p-2}  f_n^2 dx -\int_{-L+\de}^{L-\de} \Phi^{p-2}  U^2 dx|\leq \|\Phi\|_{L^\infty}^{p-2} \|f_n-U\|_{L^2[-L+\de, L-\de]} (\|f_n\|_{L^2}+\|U\|_{L^2}).
  $$
  Putting all this together ensures \eqref{220}. 
  
  We now claim that $U$ is not identically zero. Assume for a contradiction that $U=0$. In view of \eqref{240} and \eqref{220}, this implies that $\lim_n \|g_n'\|_{L^2}=0$. Since $g_n\in H^1[-L,L]$, it follows that it is uniformly continuous function, whence 
  the limit $\lim_{x\to L-} g_n(x)=\lim_{x\to L-} \Phi(x) f_n(x)=:c_n$ exists. Then, $c_n=0$, since otherwise $|f_n(x)|\geq \f{|c_n|}{2\Phi(x)}\geq \f{C_n}{(L-x)}$  for all $x\in (L-\de, L)$ and some $C_n>0$. But then, 
  $$
  1=\int_{-L}^L f_n^2(x) dx \geq \int_{L-\de}^L f_n^2(x) dx\geq C_n^2 
  \int_{L-\de}^L \f{1}{(L-x)^2}  dx=\infty,
  $$
  a contradiction. It follows that $g_n(L)=0$. Similarly, $g_n(-L)=0$. 
 Now, we can estimate 
  $$
  \int_{-L}^0 |f_n(x)|^2 dx=\int_{-L}^0 \f{|g_n(x)|^2}{\Phi^2(x)} dx \leq C  \int_{-L}^0 \f{|g_n(x)|^2}{(L+x)^2} dx\leq C \int_{-L}^0 |g_n'(x)|^2 dx, 
  $$
  where we have first observed that $\Phi(x)> C(L+x)$ on $x\in (-L,0)$, and  in the last step, we have applied the Hardy's inequality (see \eqref{har}), as $g_n(-L)=0$. Similarly, 
  $$
  \int_{0}^L |f_n(x)|^2 dx\leq C  \int_{0}^L |g_n'(x)|^2 dx.
  $$
  Combining the last two inequalities, we obtain 
  $$
  1=\int_{-L}^0 |f_n(x)|^2 dx+\int_{0}^L |f_n(x)|^2 dx\leq C \int_{-L}^L |g_n'(x)|^2 dx,
  $$
  which is contradictory as $\lim_n \|g_n'\|_{L^2}=0$. This proves that $U$ is not identically zero. 
  
  As $f_n=P_{[0,1/n)} f_n$ and $f_n\rightharpoonup U$, it is clear that  $U$ is an eigenfunction for $\cl_+$. This can also be seen as a consequence of \eqref{220}, the inequality $\liminf_n \|g_n'\|_{L^2}\geq \|g\|_{L^2} $ (which is the lower semi-continuity of $L^2$ norm, with respect to weak convergence) and \eqref{240}.
  
   In any case, we conclude  that $\cl_+ U=0$. Recall that $f_n\perp \Phi'$, whence their weak limit $f_n\rightharpoonup U$ also satisfies $U\perp \Phi'$. According to \eqref{210} however, this implies that 
   	$$
   	L_+(\sqrt{\Phi} U)=0,
   	$$
  at least in a distributional sense\footnote{This is due to the presence of the exponentially growing in the spatial variable  factor $\Phi^{-1/2}$ in the formula \eqref{210}.}, against  the compactly supported test functions. Standard elliptic theory, together with the decay properties of $\sqrt{\Phi} U$ proves that $\sqrt{\Phi} U$ is indeed an $L^2$ eigenfunction for $L_+$.  In addition, due to $\cl_+[\Phi']=0$ and  again \eqref{210}, we also have $L_+[\sqrt{\Phi} \Phi' ]=0$ as well.  According to the standard Sturm-Liouville theory for Schr\"odinger operators acting on $\rone$, with exponentially decaying potentials, each eigenvalue of $L_+$  is simple. In our case however, we have two candidates for eigenfunctions corresponding to the zero eigenvalue, namely $\sqrt{\Phi} U, \sqrt{\Phi} \Phi'$. So, it must be that 
  $\sqrt{\Phi} U=const. \sqrt{\Phi} \Phi'$, or $U=const. \Phi'$. This is in turn  contradictory as $U\perp \Phi'$ and $U$ is not identically zero. This concludes the proof of Proposition \ref{prop:45}. 
  \end{proof}
  \subsection{Spectral theory for the operator $\cl_-$}
  In order to analyze the Schr\"odinger eigenvalue problem \eqref{412}, we shall need also basic properties of the spectrum of $\cl_-$. In the classical theory, such operator is non-negative, with a simple eigenvalue at zero. The same results holds here as well. 
  
  We start however with a formula for $\cl_-$ in the spirit of \eqref{210}. Namely, for the Schr\"odinger operator 
  $$
  L_-  = -\p_t^2+\f{9 \om}{4}  - \f{3(p+1)}{2p} c(\om) \Phi^{p-2} 
  $$
  defined on $L^2(\rone)$, there is the relation 
  \begin{equation}
  \label{814} 
 \sqrt{\Phi} \cl_- f= L_- (\sqrt{\Phi} f). 
  \end{equation}
  \begin{proposition}
  \label{L_minus} 
 Let $\cl_-$ be any self-adjoint extension of the symmetric operator
 $$
  \cl_- f=-\Phi\p_x^2[\Phi f]+2(\om-c(\om) \Phi^{p-2}) f,
 $$
 defined through the quadratic form $q(u,v)=\dpr{\Phi u}{\Phi v}+2\dpr{(\om-c(\om) \Phi^{p-2}) u}{v}$ for $u,v\in C^\infty_0(-L,L)$. 
 
  Then,  $\cl_-\geq 0$, where zero is a simple eigenvalue, with $Ker[\cl_-]=span[\Phi]$. 
  \end{proposition}
  \begin{proof}
  	A direct inspection shows that 
  	$$
  	\cl_-[\Phi]=-\Phi\p_x^2[\Phi^2]+2 (\om-c(\om) \Phi^{p-2})\Phi=0, 
  	$$
  	due to the profile equation  \eqref{60}. 
  	
  	Next, we show that there is the point-wise domination $\cl_-\geq \cl_+$, as in the classical case. Indeed,  we have  
  	\begin{eqnarray*}
  	\cl_- - \cl_+ &=& \f{9 \om}{4}-\f{3(p+1)}{2p}c(\om) \Phi^{p-2} -\left(  \f{ \om}{4}-\f{2 p^2-5 p+3}{2p}c(\om) \Phi^{p-2} \right) = \\
  	&=& 2\om - (4-p) c(\om) \Phi^{p-2}
  	\end{eqnarray*}
  	If $p\geq 4$, this is clearly a positive quantity.   In the case $2<p<4$, taking into account that $ \Phi\leq\Phi(0)=\left(\f{p \om}{2c(\om)}\right)^{\f{1}{p-2}}$, we conclude again 
  	$$
  	\cl_- - \cl_+\geq 2\om-\f{(4-p)p\om}{2}=\f{\om}{2}(p-2)^2 > 0
  	$$
  	As it was already shown in Proposition \ref{prop:23}, that $\cl_+|_{\{\Phi^{p-1}\}^\perp}\geq 0$, it follows that $\cl_-|_{\{\Phi^{p-1}\}^\perp}\geq 0$. In addition, such an inequality guarantees that $\cl_-$ may have at most a single negative eigenvalue at the bottom of its spectrum. We proceed to rule this out. Recalling 
  	 the relationship \eqref{814}, it follows that if $\cl_-$ has a negative eigenvalue, then the operator  $L_-$ has negative eigenvalue as well. Thus, it remains to rule out negative eigenvalues for $L_-$. 
  	
  	Recall that since $\cl_-[\Phi]=0$, by \eqref{814}, it follows that $L_-[\Phi^{\f{3}{2}}]=0$. Thus, the Schr\"odinger operator $L_-$ has an eigenvalue at zero, with corresponding  positive eigenfunction $\Phi^{\f{3}{2}}$. By Sturm-Liouville's theorem, this means that zero is at the bottom of the spectrum for $L_-$. We have thus ruled out negative eigenvalues for $L_-$, whence $\cl_-\geq 0$. 
  	
  	Finally, suppose that $\Psi$ is an eigenfunction for $\cl_-$, corresponding to the zero eigenvalue. That is, $\cl_-[\Psi]=0$. By \eqref{814}, $L_-[\Psi \sqrt{\Phi}]=0$. As we have just seen, zero is at the bottom of the spectrum for $L_-$ and it is hence a simple eigenvalue, with a corresponding eigenfunction $\Phi^{\f{3}{2}}$. It follows that 
  	$\Psi \sqrt{\Phi}=const. \Phi^{\f{3}{2}}$ or $\Psi=const. \Phi$. Thus $Ker[\cl_-]=span[\Phi]$ and the proof of Proposition \ref{L_minus} is complete. 
  	
  \end{proof}
 
  \section{Spectral stability of the compacton waves} 
  Our next task is to study  the stability of the waves $\phi$, which satisfy \eqref{40}. Before we address these issues, for both the degenerate NLS and KdV cases, we would like to comment on the precise relation between the solutions $\vp$  to the variational problem \eqref{34} and the waves $\phi$. 
  \begin{proposition}
  	\label{prop:49}
  	The profile equation \eqref{40} has unique bell-shaped solution $\phi:[-L,L]\to\rone$
  Moreover, this solution is related to the unique solution $\Phi$ of the variational problem \eqref{30} via a  formula 
  	\begin{equation}
  		\label{340} 
  		\Phi_\om(x)= c_{p,\om} \om^{-\al} \phi(\om^{\al} x), \al=\f{p+4}{2(p+1)(p-2)}. 
  	\end{equation}
  Finally, 
  \begin{equation}
  	\label{l:49}
  	\Phi_\om(x)=\om^{\frac{1}{2(p+1)}}\Phi_1(\om^{\f{p}{2p+2}} x),
  \end{equation}
where $\Phi_1$ is the unique minimizer of \eqref{30} with $\om=1$. 
  \end{proposition}
  \begin{proof}
  	We have already discussed the uniqueness of $\vp$, and consequently of $\Phi=\sqrt{\vp}$, see the remarks after Proposition \ref{prop:22}. It remains to note that due to the relation  \eqref{300}, we have that  $c(\om,p)=const(p)  \om^{\f{p+4}{2(p+1)}}$, and so plugging in the relation \eqref{340} in the Euler-Lagrange equation \eqref{30} yields \eqref{40}, for appropriately chosen constant $c_{\om,p}$. Note that this constant $c_{\om,p}$  can be derived explicitly based on \eqref{300}, but we will not do so herein. 
  	
  	The formula for $\Phi_\om$ in terms of $\Phi_1$ is due to  an elementary scaling transformation, which transforms the variational problem \eqref{30} for general $\om>0$, into the one for $\om=1$. 
  \end{proof}
  Based on the results of Proposition \ref{prop:49}, we can claim  the following  properties of the operators $\ch_\pm$, based on the corresponding $\cl_\pm$. 
  \begin{proposition}
  	\label{prop:46} 
  	Any self-adjoint extensions of the symmetric operators $\ch_\pm$ satisfy the following properties 
  	\begin{itemize}
  		\item $\ch_-\geq 0$, with a simple eigenvalue at zero, given by $\Phi$, i.e. 
  		$$
  		Ker(\ch_-)=span[\phi], \ch_-|_{\{\phi\}^\perp}>0.
  		$$
  		\item The operator $\ch_+$ has exactly one negative eigenvalue, the second smallest eigenvalue is zero, which is also simple. In fact, $n(\ch_+)=1$, while $Ker(\ch_+)=span[\phi']$. 
  	  	\end{itemize}
  \end{proposition}
  We are now ready to proceed with the analysis of the spectral stability of the compacton waves $\phi$. 
  We start with the degenerate NLS case. 
  \subsection{Spectral stability of the degenerate NLS compactons}
  According to the instability index theory developed in Section \ref{sec:2.3}, we start with $Ker(\ch)$. Clearly, 
  $$
  Ker(\cj\ch)=Ker(\ch)=span[\left(\begin{array}{c} Ker(\ch_+)\\ 0 \end{array} \right),  \left(\begin{array}{c} 0 \\ Ker(\ch_-) \end{array} \right)]=span[\left(\begin{array}{c} \phi'\\ 0 \end{array} \right), \left(\begin{array}{c} 0 \\ \phi \end{array} \right)].
  $$
  Looking for adjoints, we solve $\cj\ch \vec{f}= \left(\begin{array}{c} \phi'\\ 0 \end{array} \right)$, which yields 
  $$
  \vec{f}=-\left(\begin{array}{c} 0\\ \ch_-^{-1} \phi' \end{array} \right).
  $$
  Same as in the classical cases, further adjoints are impossible, behind $\vec{f}$. Indeed, assuming $\cj \ch \vec{g}=\vec{f}$, we need to solve $\ch_+ g_1=-\ch_-^{-1} [\phi']$. This is however a contradiction by Fredholm theory, since 
  $$
  0=\dpr{\ch_+ g_1}{\phi'}=-\dpr{\ch_-^{-1} [\phi']}{\phi'}<0,
  $$
  due to the fact that $\ch_-^{-1}|_{\{\phi\}^\perp}>0$. 
  
  On the other hand, solving $\cj\ch \vec{f}= \left(\begin{array}{c} 0\\ \phi \end{array} \right)$, produces an adjoint 
  $$
  \vec{f}=  \left(\begin{array}{c} \ch_+^{-1} \phi\\ 0 \end{array} \right).
  $$
  Looking for further adjoints involves the quation $\cj\ch \vec{g}= \left(\begin{array}{c} \ch_+^{-1} \phi\\ 0 \end{array} \right)$, which results in \\ $\ch_- g_2=-\ch_+^{-1} \phi$. By Fredholm theory, this requires a solvability condition $\dpr{\ch_+^{-1} \phi}{\phi}=0$. Thus, we may conclude that as long as $\dpr{\ch_+^{-1} \phi}{\phi}\neq 0$, 
  $$
  gKer(\cj\ch)\ominus Ker(\ch)=span[\left(\begin{array}{c} \ch_+^{-1} \phi\\ 0 \end{array} \right)].
  $$
  According to the instability index theory, the matrix $D$ is one dimensional and the stability is determined by the sign of  
  $\dpr{\ch_+^{-1} \phi}{\phi}$. 
  
  It turns out that this is a quantity easy to compute. In fact, from the results of Proposition \ref{prop:49}, we can see that the mapping $\om\to \phi_\om$ is a $C^1$ mapping from the positive line into the set of functions. This allows us to take a Frechet derivative with respect to $\om$ in the profile equation \eqref{40}, just as in the classical case. In short, we obtain 
  $$
  \ch_+[\p_\om \phi_\om]=-\phi, 
  $$
  whence 
  $$
  \dpr{\ch_+^{-1} \phi}{\phi}=-\dpr{\p_\om, \phi_\om}{\phi_\om}=-\f{1}{2} \p_\om \int \phi^2 dx.
  $$
   \subsection{Spectral stability for  the degenerate KdV waves }
   We use, again,  the procedure,  outlined in Section \ref{sec:2.3}.   As established in Proposition \ref{prop:46}, 
   the eigenvalue problem \eqref{Lin} clearly has a one dimensional kernel, namely  $Ker(\ch_+)=span[\phi']$. We now proceed to find the generalized kernel of $\p_x \ch_+$. So, we need a $\psi\in D(\p_x \ch_+)\subset L^2(\rone)$, so that $\p_x\ch_+ \psi=\phi'$. 
   It follows that 
   \begin{equation}
   	\label{400} 
   	\ch_+\psi=\phi+c, x\in \rone .
   \end{equation}
  Taking into account the specific form of the operator $\ch_+$, in particular \eqref{H}, we conclude that for $|x|>L$, one must have 
  $$
 \ch_+\psi(x)= \om \psi(x)=c, |x|>L,
  $$
  since $supp \phi\subset (-L,L)$. 
  This is of course impossible, as $\psi\in L^2(\rone)$, unless 
    $c=0$. Thus, \eqref{400} becomes $\ch_+ \psi=\phi$. 
    
    Noting that $\phi\perp Ker(\ch_+)$, so that $\ch_+^{-1} \phi$ exists (uniquely in the subspace $Ker(\ch_+)^\perp$),  it follows that $\psi\in \ch_+^{-1} \phi+Ker(\ch_+)$, so we may select 
  $$
  \psi=\ch_+^{-1} \phi,
  $$
   and this is the unique element generating the subspace $gKer(\p_x \ch_+) \ominus Ker(\ch_+)$. 
   It follows that the stability of the traveling wave $\phi(x- \om t)$, once again is equivalent to the condition $\dpr{\ch_+^{-1} \phi}{\phi}<0$. 
   
   \subsection{Computation of $\dpr{\ch_+^{-1} \phi}{\phi}$}
   Since $\phi:(0, \infty)\to (0, \phi_0)$ is a bijection and based on the representation \eqref{k:129}, we have 
   \begin{eqnarray*}
   	\int_{-\infty}^\infty \phi^2 dx &=& 2\int_0^\infty \phi^2 dx = -2\int_0^{\infty} \f{\phi^2}{\phi'} d x=
   	2\int_0^{\phi_0} \f{\phi^2}{\sqrt{\om-\f{2}{p}\phi^{p-2}}} d\phi=\\
   	&=& 
   	2\phi_0^3 \int_0^1 \f{z^2}{\sqrt{\om-\f{2}{p} \phi_0^{p-2} z^{p-2}}} dz
   	= 2\left(\f{p}{2}\right)^{\f{3}{p-2}} \om^{\f{3}{p-2}-\f{1}{2}} \int_0^1 \f{z^2}{\sqrt{1-z^{p-2}}} dz,
   \end{eqnarray*}
   so that 
   $$
   D =-\f{1}{2} \p_\om 	\int_{-\infty}^\infty \phi^2 dx=   -\left(\f{p}{2}\right)^{\f{3}{p-2}} \f{(8-p)}{2(p-2)}  \om^{\f{3}{p-2}-\f{3}{2}} \int_0^1 \f{z^2}{\sqrt{1-z^{p-2}}} dz,
   $$
This calculation   yields the necessary and sufficient condition\footnote{The point $p=8$ is the threshold for stability, similar to the case $p=5$ for the standard NLS. At this value, if we use $p$ as a bifurcation parameter, 
		going from $p<8$ to $p>8$, there is a crossing of a pair of purely imaginary eigenvalues through zero 
	 	to a real pair of eigenvalues, one positive and one negative }  for stability $2<p\leq 8$.

\end{document}